\documentclass[11pt, leqno]{amsart}
\usepackage{amsmath}
\usepackage{amsfonts}
\usepackage{amssymb}
\usepackage{dsfont}
\usepackage{graphicx}
\usepackage{mathrsfs}
\usepackage[margin=2.90cm]{geometry}
\usepackage{url}
\usepackage{hyperref}
\providecommand{\U}[1]{\protect\rule{.1in}{.1in}}
\newtheorem{theorem}{Theorem}[section]
\newtheorem*{acknowledgement*}{Acknowledgements}

\newtheorem{corollary}[theorem]{Corollary}

\newtheorem{lemma}[theorem]{Lemma}

\newtheorem{remark}[theorem]{Remark}

\def\<{\left\langle}
\def\>{\right\rangle}

\newcommand{\matrice}{\begin{pmatrix}}
\newcommand{\ok}{\end{pmatrix}}
\newcommand{\dmatrice}{\begin{vmatrix}}
\newcommand{\dok}{\end{vmatrix}}

\def\<{\left\langle}
\def\>{\right\rangle}

\begin{document}
\title[Index of $f$-minimal hypersurfaces: general ambients]{Index and first Betti number of $f$-minimal hypersurfaces: general ambients}


\author[Debora Impera]{Debora Impera}
\address[Debora Impera]{Dipartimento di Scienze Matematiche "Giuseppe Luigi Lagrange", Politecnico di Torino, Corso Duca degli Abruzzi, 24, Torino, Italy, I-10129}
\email{debora.impera@gmail.com}

\author[Michele Rimoldi]{Michele Rimoldi}
\address[Michele Rimoldi]{Dipartimento di Scienze Matematiche "Giuseppe Luigi Lagrange", Politecnico di Torino, Corso Duca degli Abruzzi, 24, Torino, Italy, I-10129}
\email{michele.rimoldi@gmail.com}


\begin{abstract}
We generalize a method by L. Ambrozio, A. Carlotto, and B. Sharp to study the Morse index of closed $f$-minimal hypersurfaces isometrically immersed in a general weighted manifold. The technique permits, in particular, to obtain a linear lower bound on the Morse index via the first Betti number for closed $f$-minimal hypersurfaces in products of some compact rank one symmetric spaces with an Euclidean factor, endowed with the rigid shrinking gradient Ricci soliton structure. These include, as particular cases, all cylindric shrinking gradient Ricci solitons.
\end{abstract}

\maketitle
\tableofcontents

\section{Introduction}

Throughout this note, we consider as ambient manifold an $(m+1)$-dimensional weighted manifold $M_{f}=(M, g, e^{-f}d\mathrm{vol}_{M})$, namely a Riemannian manifold $(M^{m+1}, g)$ endowed with a measure with smooth positive density $e^{-f}$ with respect to the Riemannian volume measure $d\mathrm{vol}_{M}$. We are interested in closed isometrically immersed orientable $f$-minimal hypersurfaces  $\Sigma$ in $M_{f}$, namely critical points of the weighted volume functional
\[
\ \mathrm{vol}_{f}(\Sigma)=\int_{\Sigma}e^{-f}d\mathrm{vol}_{\Sigma}. 
\]
These are exactly those immersions with vanishing $f$-mean curvature
\[
\ H_{f}\doteq H+\frac{\partial f}{\partial N}=0.
\]
Here $H=\mathrm{tr}A$, and $A$ is the shape operator of the immersion, defined as $AX=-\nabla^{M}_{X}N$ on all tangent vector $X$ to $\Sigma$.

The connection between the geometry of the ambient weighted manifold and $f$-minimal hypersurfaces occurs via the second variation of the weighted area functional. Given a $f$-minimal hypersurface and a normal variational vector field $uN$, this yields the quadratic form
\[
\ Q_{f}(u,u)=\int_{\Sigma}\left(|\nabla u|^2-\left(\mathrm{Ric}_{f}^{M}(N, N)+|A|^2\right)u^{2}\right)e^{-f}d\mathrm{vol}_{\Sigma},
\]
where $\mathrm{Ric}^{M}_{f}=\mathrm{Ric}^{M}+\mathrm{Hess}^{M}(f)$ denotes the Bakry-\'Emery Ricci tensor of the ambient weighted manifold $M_f$. It follows that stability properties of $f$-minimal hypersurfaces are taken into account by spectral properties of the following weighted Jacobi operator
\[
\ L_f=\Delta_f-\left(\mathrm{Ric}^{M}_{f}(N,N)+|A|^2\right),
\]
whose leading term is the $f$-Laplacian $\Delta_f=\Delta+\left\langle \nabla f, \nabla\,\right\rangle$ on $\Sigma_f$. Note that we are using the convention that $\Delta\doteq-\mathrm{div}\left(\nabla\right)$. Dealing with closed smooth hypersurfaces, the operator $L_{f}$ is self-adjoint with respect to the measure $e^{-f}d\mathrm{vol}_{\Sigma}$ and elliptic, and hence has discrete spectrum which tends to infinity. We then define the $f$-index of $\Sigma$ as
\[
\mathrm{Ind}_f(\Sigma)=\sharp\{\textrm{negative eigenvalues of }L_f\},
\]
i.e. the maximal dimension of the space on which the index form $Q_{f}$ is negative definite.

For a more detailed discussion about stability properties and estimates on the $f$-index of $f$-minimal hypersurfaces we refer the reader to the introduction in \cite{IR1} and \cite{IRS} and references therein.

\subsection{Topological complexity implies high instability}The recent impressive developments in the existence theory for minimal immersions through min-max method have motivated a renewed interest in studying estimates on the Morse index of minimal immersions. 

One possible way to control instability is through topological invariants (in particular the first Betti number) of the minimal hypersurface. This was first investigated by A. Ros in \cite{Ros} for immersed minimal surfaces in $\mathbb{R}^{3}$, or a quotient of it by a group of translations, and then, in higher dimension, by A. Savo when the ambient manifold is the round sphere, \cite{Savo}. The idea behind these works is to use harmonic $1$-forms and some ambient structure to make some interesting functions out of those, permitting to get a lower bound on the index. This technique was recently extended and generalized by L Ambrozio, A. Carlotto and B. Sharp, \cite{ACS}, who showed, from a general perspective, that the Morse index is bounded from below by a linear function of the first Betti number for all closed minimal hypersurfaces in a large class of positively curved ambient manifolds. This comprises, for instance, all compact rank one symmetric spaces. In their work, the ambient structure is taken into account by the existence of some nice isometric immersion into an Euclidean space $\mathbb{R}^{d}$ via which one can produce instability directions, again using harmonic $1$-forms. For some other recent developments see also \cite{MendRad}.

In the recent paper \cite{IRS} we related the Morse index and the first Betti number of self-shrinkers for the mean curvature flow and, more generally, of $f$-minimal hypersurfaces in a weighted Euclidean space endowed with a convex weight. Following the ideas adopted in \cite{IRS} and motivated by the approach introduced in \cite{ACS},  in this short note we want to study the same problem in the wider setting of $f$-minimal hypersufaces sitting in a general weighted manifold.

\subsection{Main $f$-index estimate}As in \cite{IRS}, in order to get our results we will produce good test-functions for the quadratic form $Q_{f}$ using $f$-harmonic 1-forms of the $f$-minimal hypersurface. Recall that $f$-harmonic 1-forms on compact hypersurfaces are exactly those 1-forms $\omega$ which are simultaneously closed and $f$-coclosed, i.e,
\[
\ d\omega=\delta_{f}\omega=0,
\]
with $\delta_{f}\omega=\delta\omega+\omega(\nabla f)$ the weighted codifferential. 
\medskip

Our main $f$-index estimate will follow from the following concentration of the spectrum inequality.

\begin{theorem}\label{mainthmfminimalgen}
Let $M_{f}$ be am $(m+1)$-dimensional Riemannian manifold that is isometrically embedded in some Euclidean space $\mathbb{R}^{d}$, $d\geq m+1$. Let $\Sigma^{m}$ be a closed isometrically immersed $f$-minimal hypersurface of $M^{m+1}$. Assume that there exists a real number $\eta$ and a $q$-dimensional vector space $\mathcal{V}$ of $f$-harmonic 1-forms on $\Sigma^m$ such that for any $\omega\in \mathcal{V}\backslash \{0\}$,
\begin{align}
&\int_{\Sigma}\left(\sum_{k=1}^{m}\left|II(e_{k},\omega^\sharp)\right|^2+\sum_{k=1}^{m}\left|II(e_{k},N)\right|^2|\omega|^2\right)e^{-f}d\mathrm{vol}_{\Sigma}\label{eq_intcurvassumpgen}\\
<&\int_{\Sigma}\left(\mathrm{Ric}_{f}^{M}(N, N)|\omega|^2+\,\mathrm{Ric}_{f}^{M}(\omega^\sharp,\omega^\sharp)-\left\langle\,R^{M}(\omega^\sharp,N)\omega^\sharp,N\right\rangle\right)e^{-f}d\mathrm{vol}_{\Sigma}+\eta\int_\Sigma|\omega|^2e^{-f}d\mathrm{vol}_{\Sigma}, \nonumber
\end{align}
where $II$ denotes the second fundamental form of $M^{m+1}$ in $\mathbb{R}^{d}$, $\left\{e_{1},\ldots, e_{m}\right\}$ is a local orthonormal frame on $\Sigma$, and $N$ is a unit normal vector field on $\Sigma$. Then
\[
\sharp\{\textrm{eigenvalues of $L_f$}<\eta\}\geq \frac{2}{d(d-1)}q.
\]
\end{theorem}

As a corollary, when $\eta=0$, we obtain the following $f$-index estimate. Recall that by the Hodge decomposition in the weighted setting (see \cite{Bue}) the dimension of the space of $f$-harmonic $1$-forms still equals the first Betti number $b_{1}(\Sigma)$ of $\Sigma$.

\begin{corollary}\label{maincorfminimal}
Let $M_{f}$ be an $(m+1)$-dimensional Riemannian manifold that is isometrically embedded in some Euclidean space $\mathbb{R}^{d}$, $d\geq m+1$. Let $\Sigma^{m}$ be a closed isometrically immersed $f$-minimal hypersurface of $M^{m+1}$. Assume that for any non-zero $f$-harmonic $1$-form $\omega$ on $\Sigma$,
\begin{align}
&\int_{\Sigma}\left(\sum_{k=1}^{m}\left|II(e_{k},\omega^{\sharp})\right|^2+\sum_{k=1}^{m}\left|II(e_{k},N)\right|^2|\omega|^2\right)e^{-f}d\mathrm{vol}_{\Sigma}\label{eq_intcurvassump}\\&<\int_{\Sigma}\left(\mathrm{Ric}^{M}_{f}(N, N)|\omega|^2+\mathrm{Ric}^{M}_{f}(\omega^{\sharp},\omega^{\sharp})-\left\langle R^{M}(\omega^{\sharp},N)\omega^{\sharp},N\right\rangle\right)e^{-f}d\mathrm{vol}_{\Sigma}.\nonumber
\end{align}
Then
\[
\ \mathrm{Ind}_{f}(\Sigma)\geq \frac{2}{d(d-1)}b_{1}(\Sigma).
\]
\end{corollary}

\begin{remark}
\rm{It should be pointed out that P. Zhu and W. Gan, \cite{ZG} also obtained an estimate in the same spirit of Corollary \ref{maincorfminimal} above. In that paper it is proved a similar estimate on the $f$-index assuming a different curvature condition. However, this condition turns out to be less fitting to the weighted setting and difficult to verify in applications. The novelty in our approach, first considered in \cite{IRS}, is that we use $f$-harmonic forms, instead of usual harmonic forms.}
\end{remark}
%
\subsection{Applications}
Recently there has been some interest in coupling two geometric flows in order to improve geometric and analytic properties of the flows with respect to those they had as themselves.  One case of particular interest is that of the mean curvature flow of hypersurfaces in a moving ambient space.  When the ambient manifold is evolving by the Ricci flow, the coupled evolution is known as Ricci-mean curvature flow. Inspired by Huisken's monotonicity formula for hypersurfaces in the Gaussian shrinker, A. Magni, C. Mantegazza, and E. Tsatis, \cite{MMT}, and J. Lott, \cite{L}, studied monotonicity formulas for mean curvature flow in a gradient Ricci soliton background and introduced the concept of mean curvature soliton which can be regarded as a generalization of self-shrinkers for the mean curvature flow in the Euclidean space. By its definition, a mean curvature soliton is nothing but an $f$-minimal hypersurface isometrically immersed in a gradient Ricci soliton with $f$ being the potential function of the ambient soliton. This fact motivated some recent works on $f$-minimal hypersurfaces in gradient Ricci solitons; see e.g. \cite{CZ}, \cite{CMZ}, where it is considered the particular case of $f$-minimal hypersurfaces in the cylindric shrinking gradient Ricci soliton with one dimensional Euclidean factor $(\mathbb{S}^{m}(\sqrt{(m-1)/\lambda}) \times\mathbb{R}, g, \nabla f)$, with $g$ the product metric, $f(x,t)=\frac{\lambda}{2}t^2$, and $\lambda>0$.
\medskip

As a first application of our main result, we can obtain the following result concerning $f$-minimal hypersurfaces in all cylindric shrinking gradient Ricci solitons. In the case of $f$-minimal hypersurfaces in the cylindric shrinking gradient Ricci soliton with one-dimensional Euclidean factor, a similar estimate was obtained in \cite{ZG}, but assuming an additional curvature assumption.

\begin{corollary}\label{coroindcyl1}  
Let $\Sigma$ be a closed isometrically immersed $f$-minimal hypersurface in the weighted manifold $M_{f}^{m+1}=\left(\mathbb{S}^{k}(\sqrt{\frac{(k-1)}{\lambda}})\times\mathbb{R}^{m-k+1},g_M, e^{-f}d\mathrm{vol}_M\right)$, where $2\leq k\leq m$, $\lambda>0$, $g_M=g_{\mathbb{S}^{k}(\sqrt{\frac{(k-1)}{\lambda}})}+dt_1^2+\dots+dt_{m-k+1}^2$, $g_{\mathbb{S}^{k}(\sqrt{\frac{(k-1)}{\lambda}})}$ denotes the canonical round metric on the sphere, $d\mathrm{vol}_M$ the Riemannian volume measure, and $f(x,t)=\frac{\lambda}{2}|t|^2$. Then
\begin{equation}\label{Estcoro}
\mathrm{Ind}_{f}(\Sigma)\geq \frac{2}{(m+2)(m+1)}b_{1}(\Sigma).
\end{equation}
\end{corollary}

\begin{remark}
\rm{For closed $f$-minimal hypersurfaces in the cylindric shrinking gradient Ricci soliton with one-dimensional Euclidean factor it was proved in \cite{CMZ} that $\mathrm{Ind}_{f}(\Sigma)\geq 1$ and, indeed, equality holds if and only if $\Sigma=\mathbb{S}^{m-1}\times\left\{0\right\}$. When the first Betti number is large our estimate \eqref{Estcoro} improves this result.}
\end{remark}

\begin{remark}
\rm{Note that in some special situations shrinking gradient Ricci solitons are necessarily of cylindric type. For instance we recall that H. D. Cao,  B.-L. Chen and X.-P Zhu, \cite{CCZ} proved that any $3$-dimensional complete non-compact non-flat shrinking gradient Ricci soliton is a finite quotient of the round cylinder $\mathbb{S}^{2}\times\mathbb{R}$. In higher dimension Z.-H. Zhang proved that any locally conformally flat complete non-compact gradient shrinking Ricci soliton must be a finite quotient of the Gaussian shrinking Ricci soliton $\mathbb{R}^{n}$ or the cylindric shrinking gradient Ricci soliton $\mathbb{S}^{n-1}\times\mathbb{R}$ ; \cite{Zh}.}
\end{remark}

More generally, one could apply Corollary \ref{maincorfminimal} to obtain index bounds for $f$-minimal hypersurfaces in products of almost all other compact rank one symmetric spaces (CROSSes) with an Euclidean factor, endowed with the rigid gradient Ricci soliton structure. Note that the compact  rank one symmetric spaces, i.e. all compact symmetric spaces of positive sectional curvature, are the spheres, $\mathbb{R}\mathbb{P}^{k}$, $\mathbb{C}\mathbb{P}^n$ with $2n=k$, $\mathbb{H}\mathbb{P}^p$ with $4p=k$ and $Ca\mathbb{P}^2$.

\begin{corollary}\label{coroindCROSSxR}
Let $P$ be either $\mathbb{C}\mathbb{P}^n$ with $k=2n\geq 4$, $\mathbb{H}\mathbb{P}^p$ with $k=4p$ or $Ca\mathbb{P}^2$. Let $\Sigma$ be a closed isometrically immersed $f$-minimal hypersurface in the weighted manifold $M_{f}^{m+1}=\left(P^k\times\mathbb{R}^{m-k+1},g_M, e^{-f}d\mathrm{vol}_M\right)$, $g_M=g_{P}+dt_1^2+\dots+dt_{m-k+1}^2$, with $g_{P}$ denoting the canonical metric on $P$, $d\mathrm{vol}_M$ the Riemannian volume measure, and $f(x,t)=\frac{\lambda}{2}|t|^2$. Then 
\[
\ \mathrm{Ind}_{f}(\Sigma)\geq \frac{2}{d(d-1)}b_{1}(\Sigma),
\]
where
\begin{itemize}
\item $d=m+2+2n^2$ if $P=\mathbb{C}\mathbb{P}^n$, $2n=k$;
\item $d=m+2+2p^2-p$ if $P=\mathbb{H}\mathbb{P}^p$, $4p=k$;
\item $d=m+12$ if $P=Ca\mathbb{P}^2$.
\end{itemize}
\end{corollary}

\begin{remark}
\rm{Corollary \ref{maincorfminimal} does not seem to apply directly to the case $P=\mathbb{RP}^{k}$, since in this case we cannot produce an immersion of $P\times\mathbb{R}^{m+k-1}$ satisfying the integral curvature assumption \eqref{eq_intcurvassump}. This difficulty is already encountered in the classical non-weighted case but, in that case, one could overcome this by using the correspondence between immersed minimal hypersurfaces in the real projective space and immersed hypersurfaces of the sphere that are invariant under the antipodal map; see \cite{ACS}. However, this kind of reasoning seems difficult to apply in our setting. More generally, it would be interesting to prove an analogous index estimate for closed $f$-minimal hypersurfaces in all the so called \textsl{rigid} shrinking Ricci solitons i.e. quotients $N\times_{\Gamma}\mathbb{R}^{k}$ with $N$ Einstein with Einstein constant $\lambda$ and $\Gamma$ acting freely on $N$ and by orthogonal transformations on $\mathbb{R}^{k}$, with $f=\frac{\lambda}{2}d^2$ and $d$ is the distance in the flat fibers to the base; see \cite{PW_Rig}. The obvious difficulty is that in general we do not have such a good control on the embedding}.
\end{remark}

In another direction, a further situation in which we can apply Corollary \ref{maincorfminimal} is that in which the ambient weighted manifold is a convex hypersurface of the Euclidean space with sufficiently pinched principal curvatures. Indeed minor modifications to the proof of Theorem 12 in \cite{ACS} permit to obtain the following result.

\begin{corollary}\label{CorPinched}
Let $M^{m+1}_{f}$ be a weighted manifold isometrically embedded as an hypersurface in $\mathbb{R}^{m+2}$ in such a way that the principal curvatures $k_{1}\leq\ldots\leq k_{m+1}$ with respect to the outward pointing unit normal $\nu$ are positive and satisfy the pinching condition
\[
\ \frac{k_{m+1}}{k_{1}}<\sqrt{\frac{m+1}{2}}.
\]
Assume that $\mathrm{Hess}^{M}f>0$. Then every closed embedded minimal $f$-hypersurface $\Sigma^m$ of $M^{m+1}_{f}$ is such that
\[
\ \mathrm{Ind}_{f}(\Sigma)\geq \frac{2}{(m+2)(m+1)}b_{1}(\Sigma).
\]
\end{corollary}

\section{Proof of the main estimate}
%

\begin{proof}[Proof of Theorem \ref{mainthmfminimalgen}]
Let $k$ be the number of the eigenvalues of the Jacobi operator $L_f$ that are below the threshold $\eta$ and denote by $\varphi_1,\cdots,\varphi_k$ the eigenfunctions associated to the $k$ eigenvalues $\lambda_1,\cdots,\lambda_k$. 
Let $\omega$ be a $f$-harmonic $1$-form in $\mathcal{V}\subset V\doteq\mathcal{H}^{1}_{f}(\Sigma)$ and set $u_{ij}\doteq\langle N^{\flat}\wedge\omega, \theta_{ij}\rangle$, where $\{\theta_{ij}\}_{i<j}$ is an orthonormal basis of $\bigwedge^2\mathbb{R}^{d}$. Then the map
\begin{align*}
\mathcal{V} &\rightarrow \mathbb{R}^{\frac{d(d-1)}{2}\,k}\\
\omega & \mapsto \left(\int_\Sigma u_{ij}\varphi_p\,e^{-f}d\mathrm{vol}_{\Sigma}\right),
\end{align*}
for $1\leq i<j\leq d$ and $1\leq p\leq k$, is a linear map.

Assume by contradiction that $q>\frac{d(d-1)}{2}k$. Then there exists a non-zero $\omega\in \mathcal{V}$ such that
\[
\int_\Sigma u_{ij}\varphi_pe^{-f}d\mathrm{vol}_{\Sigma}=0
\]
for all $i<j$ and for all $1\leq p\leq k$. Thus, in particular
\begin{align*}
\eta\int_\Sigma |\omega|^2e^{-f}d\mathrm{vol}_{\Sigma}&\leq \lambda_{k+1}\int_\Sigma |\omega|^2e^{-f}d\mathrm{vol}_{\Sigma}\\
&=\lambda_{k+1}\sum_{i<j}^d \int_\Sigma u_{ij}^2e^{-f}d\mathrm{vol}_{\Sigma}\\
&\leq\sum_{i<j}^d Q_f(u_{ij},u_{ij}).
\end{align*}
On the other hand, given a local orthonormal frame $\{e_1,\dots, e_m\}$ on $\Sigma^m$ and denoting by $D$ the connection of $\mathbb{R}^{d}$, one has
\[
|\nabla u_{ij}|^2 =\sum_{k=1}^m |D_{e_k} u_{ij}|^2=\sum_{k=1}^m \langle D_{e_k}(N^{\flat}\wedge \omega), \theta_{ij}\rangle^2\\
\]
Hence, summing up on $i<j$,
\begin{align}
\sum_{i<j}^d |\nabla u_{ij}|^2=&\sum_{k=1}^m |D_{e_k} (N^\flat\wedge \omega)|^2.\label{Star1}
\end{align}
Note that:
\begin{align*}
|D_{e_{k}}(N^{\flat}\wedge\omega)|^2=&|D_{e_{k}}N^{\flat}\wedge\omega+N^{\flat}\wedge D_{e_{k}}\omega|^2\\
=&|D_{e_{k}}N|^2|\omega|^2+|D_{e_{k}}\omega^{\sharp}|^2-\langle\omega^{\sharp}, D_{e_{k}}N\rangle^2-\langle D_{e_{k}}\omega^{\sharp}, N\rangle^2.
\end{align*}
Recall the orthogonal decompositions:
\begin{align*}
D_{e_{k}}N=&\nabla^{M}_{e_{k}}N+II(e_{k}, N)=-Ae_{k}+II(e_{k}, N);\\
D_{e_{k}}\omega^{\sharp}=&\nabla_{e_{k}}\omega^{\sharp}+\langle A\omega^{\sharp}, e_{k}\rangle N+II(e_{k}, \omega^{\sharp}).
\end{align*}
As a consequence of the relations above one has:
\begin{align*}
|D_{e_{k}}N|^2=&|Ae_{k}|^2+|II(e_{k}, N)|^{2};\\
|D_{e_{k}}\omega^{\sharp}|^2=&|\nabla_{e_{k}}\omega^{\sharp}|^2+\langle A\omega^{\sharp}, e_{k}\rangle^{2}+|II(e_{k}, \omega^{\sharp})|^2;\\
\langle D_{e_{k}}N, \omega^{\sharp}\rangle=&-\langle A\omega^{\sharp}, e_{k}\rangle;\\
\langle D_{e_{k}}\omega^{\sharp}, N\rangle=&\langle A\omega^{\sharp}, e_{k}\rangle.
\end{align*}
Inserting the previous relations in \eqref{Star1},
\begin{align}
\sum_{i<j}^{d}|\nabla u_{ij}|^2=&|\nabla \omega|^2-|A\omega^\sharp|^2+|A|^2|\omega|^2\\
&+\sum_{k=1}^m|II(e_k,\omega^\sharp)|^2+\sum_{k=1}^m |II({e_k},N)|^2|\omega|^2.\nonumber
\end{align}
Thus, 
\begin{align}
\sum_{i<j}^d Q_f(u_{ij},u_{ij})=&\sum_{i<j}^d\int_\Sigma |\nabla u_{ij}|^2-( \mathrm{Ric}^{M}_f(N,N)+|A|^2)u_{ij}^2e^{-f}d\mathrm{vol}_{\Sigma}\nonumber\\
=&\int_\Sigma \sum_{k=1}^m|II(e_k,\omega^\sharp)|^2+\sum_{k=1}^m |II({e_k},N)|^2|\omega|^2 e^{-f}d\mathrm{vol}_{\Sigma} \label{eq_Qfuij}\\
&+\int_\Sigma |\nabla \omega|^2-|A\omega^\sharp|^2- \mathrm{Ric}^{M}_f(N,N)|\omega|^2e^{-f}d\mathrm{vol}_{\Sigma}\nonumber.
\end{align}

Since $\omega$ is $f$-harmonic, combining the $f$-Bochner-Weitzenb\"{o}ck formula with the Gauss equation we get
\begin{align*}
-\Delta_f \frac{|\omega|^2}{2}&=|\nabla \omega|^2+\mathrm{Ric}_f(\omega^\sharp,\omega^\sharp)\\
&=|\nabla \omega|^2+ \mathrm{Ric}^{M}_f(\omega^\sharp,\omega^\sharp)-\langle R^{M}(\omega^\sharp,N)\omega^\sharp,N \rangle-|A\omega^\sharp|^2.
\end{align*}
Integrating the previous equation and substituting in \eqref{eq_Qfuij} we thus obtain

\begin{align*}
\eta\int_\Sigma |\omega|^2e^{-f}d\mathrm{vol}_{\Sigma}\leq& \sum_{i<j}^d Q_f(u_{ij},u_{ij})\\
=&\int_\Sigma \sum_{k=1}^m|II(e_k,\omega^\sharp)|^2+\sum_{k=1}^m |II({e_k},N)|^2|\omega|^2 e^{-f}d\mathrm{vol}_{\Sigma}\\
&-\int_\Sigma\ \mathrm{Ric}^{M}_f(\omega^\sharp,\omega^\sharp)-\langle  R^{M}(\omega^\sharp,N)\omega^\sharp,N \rangle+ \mathrm{Ric}^{M}_f(N,N)|\omega|^2 e^{-f}d\mathrm{vol}_{\Sigma}.
\end{align*}

This is in contradiction with the assumption \eqref{eq_intcurvassumpgen}. The result follows.
\end{proof}

\section{$f$-minimal hypersurfaces in shrinking gradient Ricci solitons}
\subsection{Cylindric shrinking gradient Ricci solitons}
\begin{proof}[Proof of Corollary \ref{coroindcyl1}]
Let us denote by $\pi_{\mathbb{S}^{k}}$ and $\pi_{\mathbb{R}^{m-k+1}}$, respectively, the projection on the spherical factor and on the Euclidean factor of $M$, and by $\langle,\rangle$ the metric on $M$. By a straightforward computation one has that
\begin{align*}
 \mathrm{Hess}^{M}f=&\lambda (dt_1^2+\dots+dt_{m-k+1}^2),\\
 \mathrm{Ric}^{M}_f=&\lambda \langle\ ,\ \rangle,\\
\langle R^{M}(V,W)V, W\rangle=& \langle II(V,V),II(W,W)\rangle-|II(V,W)|^2,\\
II(V,W)=& \sqrt{\frac{\lambda}{k-1}}\langle (\pi_{\mathbb{S}^{k}})_*V,(\pi_{\mathbb{S}^{k}})_*W\rangle\nu,
\end{align*}
for any $V,W\in TM$, with $\nu$ normal vector field of $\mathbb{S}^{k}\times\mathbb{R}^{m-k+1}$ in $\mathbb{R}^{m+2}$. Moreover we have that
\begin{equation*}
\sum_{i=1}^{m}\langle(\pi_{\mathbb{S}^{k}})_{*}V, (\pi_{\mathbb{S}^{k}})_{*}e_{i}\rangle^2=|(\pi_{\mathbb{S}^{k}})_{*}V|^2-\langle(\pi_{\mathbb{S}^{k}})_{*}V, (\pi_{\mathbb{S}^{k}})_{*}N\rangle^2.
\end{equation*}
Hence,  for any $X\in T\Sigma$,
\begin{align*}
\sum_{i=1}^m\Big(|II(e_i,X)|^2+|II(e_i,N)|^2|X|^2\Big)=&\frac{\lambda}{k-1}\Big(|(\pi_{\mathbb{S}^{k}})_*X|^2-\langle (\pi_{\mathbb{S}^{k}})_*X,(\pi_{\mathbb{S}^{k}})_*N\rangle^2\\
&+|X|^2|(\pi_{\mathbb{S}^{k}})_*N|^2(1-|(\pi_{\mathbb{S}^{k}})_*N|^2)\Big).
\end{align*}
Moreover,
\[
\langle R^{M}(X,N)X, N\rangle=\frac{\lambda}{k-1}\Big(|(\pi_{\mathbb{S}^{k}})_*X|^2|(\pi_{\mathbb{S}^{k}})_*N|^2-\langle(\pi_{\mathbb{S}^{k}})_*X, (\pi_{\mathbb{S}^{k}})_*N\rangle^2\Big)
\]
Set 
\begin{align*}
\mathcal{R}(X,N)\doteq&\sum_{i=1}^m\Big(|II(e_i,X)|^2+|II(e_i,N)|^2|X|^2\Big)\\
&+\langle R^{M}(X,N)X, N\rangle-\mathrm{Ric}^{M}_{f}(N, N)|X|^2-\mathrm{Ric}_{f}^{M}(X,X)
\end{align*}
Then 
\begin{align*}
\mathcal{R}(X,N)=&-\frac{2\lambda(k-2)}{k-1}|X|^2-\frac{\lambda}{k-1}\Big(|X|^2|(\pi_{\mathbb{R}^{m-k+1}})_*N|^4\\
&+|(\pi_{\mathbb{R}^{m-k+1}})_*X|^2(2-|(\pi_{\mathbb{R}^{m-k+1}})_*N|^2)\\
&+2\langle (\pi_{\mathbb{R}^{m-k+1}})_*N,(\pi_{\mathbb{R}^{m-k+1}})_*X\rangle^2\Big).
\end{align*}
For $k\geq 3$ this expression is strictly negative on non-zero vector fields. In case $k=2$ the expression is non-positive. However, it vanishes if and only if $(\pi_{\mathbb{R}^{m-k+1}})_{*}N=0$, and this is in contradiction with the fact that $\Sigma$ is compact. Hence condition \eqref{eq_intcurvassump} is satisfied and the conclusion follows.
\end{proof}

\subsection{Some other rigid gradient Ricci shrinking solitons}

\begin{proof}[Proof of Corollary \ref{coroindCROSSxR}]Consider $(M^{m+1}, g_{M})$ to be of the form $M^{m+1}=P^{k}\times\mathbb{R}^{m+k-1}$, with $(P^{k}, g_{P})$ one of the projective spaces
\[
\ \mathbb{CP}^{n}\,(2n=k),\quad\mathbb{HP}^{p}\,(4p=k)\quad C\mathrm{a}\mathbb{P}^{2}\,(16=k),
\]
endowed with their standard Riemannian metrics. In the case of $\mathbb{CP}^{n}$ this is just the standard Fubini-Study metric. Note that, the following properties do hold:
\begin{itemize}
\item[(i)] Up to normalization, all these metrics have sectional curvatures bounded between $1$ and $4$;
\item[(ii)] There exists a nice family of isometric embeddings of these spaces into Euclidean spaces $\mathbb{R}^{q}$. In case of $\mathbb{CP}^{n}$, $q=(n+1)^2$, in case of $\mathbb{HP}^{p}$, $q=(p+1)(2p+1)$, and for $C\mathrm{a}\mathbb{P}^{2}$, $q=27$. For more details and references see \cite{ACS}, \cite{Chen}.
\item[(iii)] For the embeddings in point (ii) it holds that the second fundamental form $II^{P}$ satisfies
\begin{equation}\label{EmbPropr1}
|II^{P}(X,X)|^2=4|X|^4,
\end{equation}
for all vector fields $X\in TP$.
\end{itemize}
Using \eqref{EmbPropr1}, we have that for any $X, Y\in TP$
\begin{align*}
4|X-Y|^4=&|II^{P}(X-Y,X-Y)|^2=|II^{P}(X,X)+II^{P}(Y,Y)-2II^{P}(X,Y)|^2\\
=&|II^{P}(X,X)|^2+|II^{P}(Y,Y)|^2+4|II^{P}(X,Y)|^2+2\<II^{P}(X,X),II^{P}(Y,Y)\>\\
&-4\<II^{P}(X,X)+II^{P}(Y,Y),II^{P}(X,Y)\>\\
=&4|X|^4+4|Y|^4+4|II^{P}(X,Y)|^2+2\langle II^{P}(X,X),II^{P}(Y,Y)\rangle\\
&-4\langle II^{P}(X,X)+II^{P}(Y,Y),II^{P}(X,Y)\rangle^2.
\end{align*}
Analogously,
\begin{align*}
4|X+Y|^4=&4|X|^4+4|Y|^4+4|II^{P}(X,Y)|^2+2\langle II^{P}(X,X),II^{P}(Y,Y)\rangle\\
&+4\langle II^{P}(X,X)+II^{P}(Y,Y),II^{P}(X,Y)\rangle^2.
\end{align*}
Summing up,
\begin{align*}
&4(|X|^4+|Y|^4)+4|II^{P}(X,Y)|^2+2\langle II^{P}(X,X),II^{P}(Y,Y)\rangle\\=&2\left(|X-Y|^4+|X+Y|^4\right)\\
=&2\left(|X|^4+|Y|^4+4\<X,Y\>^2+2|X|^2|Y|^2-4\<X,Y\>(|X|^2+|Y|^2)+\right.\\
&\left.+|X|^4+|Y|^4+4\<X,Y\>^2+2|X|^2|Y|^2+4\<X,Y\>\left(|X|^2+|Y|^2\right)\right)\\
=&4\left(|X|^4+|Y|^4+4\<X,Y\>^2+2|X|^2|Y|^2\right).
\end{align*}
Hence, we get
\begin{equation}\label{EmbPropr2}
\langle II^{P}(X,X), II^{P}(Y, Y)\rangle+2|II^{P}(X,Y)|^2=4\left(|X|^2|Y|^2+2\<X,Y\>^2\right).
\end{equation}
Recall that by Gauss equation, for any $X, Y\in TP$, we have
\[
\ \langle R^{P}(X,Y)X, Y\rangle=\langle II^{P}(XX), II^{P}(Y,Y)\rangle-|II^{P}(X,Y)|^2.
\]
Therefore, we obtain
\begin{equation}\label{EmbPropr3}
|II^{P}(X, Y)|^2=\frac{4}{3}\left(|X|^2|Y|^2+2\<X,Y\>^2\right)-\frac{1}{3}\langle R^{P}(X,Y)X,Y\rangle.
\end{equation}
Consider now the embedding of $M^{m+1}$ in $\mathbb{R}^{d}$, $d=q+m+1-k$, obtained by trivially extending the above embeddings of $P$ in $\mathbb{R}^{q}$. Obviously, the only non-trivial part of the second fundamental form $II$ of this embedding is given by $II^{P}$. 

Let $\Sigma^m$ be a closed oriented minimal hypersurface isometrically immersed in $(M^{m+1}, g_{M})$ and, at any given point $p\in \Sigma$, consider an o.n. basis $\left\{e_{h}\right\}$ of $T_{p}\Sigma$. Let $N$ be the unit normal vector field to $\Sigma$ and, for any $U\in TM$, let us denote by $U^{P}$ the projection of $U$ on $TP$ and by $U^{R}$ the projection of $U$ on $T\mathbb{R}^{m+k-1}$. Using \eqref{EmbPropr3}, we have
\begin{equation}\label{Emb1}
\ |II(e_{h}, U)|^2=|II^{P}(e_{h}^{P}, U^{P})|^2=\frac{4}{3}\left(|e_{h}^{P}|^2|U^{P}|^2+2\<e_{h}^{P}, U^{P}\>^2\right)-\frac{1}{3}\langle R^{M}(e_{h}, U)e_{h}, U\rangle.
\end{equation}
Letting $\left\{F_{j}\right\}$ be an o.n. basis of $T_{\pi_{P}(p)}P$, we compute:
\begin{align*}
|U^{P}|^{2}\sum_{h}|e_{h}^{P}|^2=&|U^{P}|^2\sum_{h}\sum_{j=1}^{k}\<e_{h}, F_{j}\>^2\\
=&|U^{P}|^{2}\sum_{j=1}^{k}\left(|F_{j}|^2-\<F_{j}, N\>^2\right)\\
=&|U^{P}|^{2}\left(k-|N^{P}|^{2}\right);\\
2\sum_{h}\<e_{h}^{P}, U^{P}\>^{2}=&2\sum_{h}(\sum_{·j}\<e_{h}, F_{j}\>\<U, F_{j}\>)^2\\
=&2\sum_{h}\sum_{j}\<e_{h}, F_{j}\>^2\<U, F_{j}\>^2+4\sum_{h}\sum_{i\neq j}\<e_{h}, F_{j}\>\<e_{h}, F_{i}\>\<U, F_{j}\>\<U, F_{i}\>\\
=&2(\sum_{j}(|F_{j}|^2-\<F_{j}, N\>^2)\<U, F_{j}\>^2+2\sum_{i\neq j}(\<F_{i}, F_{j}\>-\<F_{i}, N\>\<F_{j}, N\>)\<U, F_{j}\>\<U, F_{i}\>)\\
=&2(\sum_{j}\<U, F_{j}\>^2-(\sum_{j}\<F_{j}, N\>^2\<U, F_{j}\>^2+2\sum_{i\neq j}\<F_{i}, N\>\<F_{j}, N\>\<U, F_{i}\>\<U, F_{j}\>))\\
=&2|U^{P}|^2-2(\sum_{j}\<F_{j}, N\>\<U,F_{j}\>)^2=2(|U^{P}|^2-\<U^{P}, N^{P}\>^2).
\end{align*}
Hence, substituting in \eqref{Emb1}, we get
\begin{equation}\label{Emb2}
\sum_{h}|II(e_{h}, U)|^2=\frac{4}{3}\left(|U^{P}|^{2}\left(k+2-|N^{P}|^2\right)-2\<U^{P}, N^{P}\>^{2}\right)-\frac{1}{3}\sum_{h}\langle R^{M}(e_{h}, U)e_{h}, U\rangle.
\end{equation}
In particular, for any $\omega$ $1$-form on $\Sigma$
\begin{align*}
\sum_{h}|II(e_{h}, N)|^2|\omega^2|=&\frac{4}{3}|\omega|^2|N^{P}|^{2}\left(k+2-3|N^{P}|^2\right)-\frac{1}{3}\mathrm{Ric}^{M}(N, N)|\omega|^2;\\
\sum_{h}|II(e_{h}, \omega^{\sharp})|^2=&\frac{4}{3}\left(|\omega^{P}|^{2}\left(k+2-|N^{P}|^2\right)-2\langle\omega^{P}, N^{P}\rangle^{2}\right)\\
&-\frac{1}{3}\mathrm{Ric}^{M}(\omega^{\sharp}, \omega^{\sharp})+\frac{1}{3}\langle R^{M}(\omega^{\sharp}, N)\omega^{\sharp}, N\rangle.
\end{align*}
Hence, using the same notations of the previous section, we have that
\begin{align}\label{CurvTermCROSSxR}
\mathcal{R}(\omega^{\sharp}, N)\doteq&\sum_{h}\left(|II(e_{h}, N)|^2|\omega|^2+|II(e_{h}, \omega^{\sharp})|^2\right)-\mathrm{Ric}^{M}(N, N)|\omega|^2-\mathrm{Ric}^{M}(\omega^{\sharp}, \omega^{\sharp})\\
&+\langle R^{M}(\omega^{\sharp}, N)\omega^{\sharp}, N\rangle-\mathrm{Hess}^{M}f(N, N)|\omega|^2-\mathrm{Hess}^{M}f(\omega^{\sharp}, \omega^{\sharp})\nonumber\\
=&\frac{4}{3}\left[|\omega|^{2}|N^{P}|^2\left(k+2-3|N^{P}|^2\right)+|\omega^{P}|^2\left(k+2-|N^{P}|^2\right)\right.\nonumber\\
&\left.\,\,\,\,\,\,-2\langle\omega^{P}, N^{P}\rangle^2-\mathrm{Ric}^{M}(N, N)|\omega|^2-\mathrm{Ric}^{M}(\omega^{\sharp}, \omega^{\sharp})+\langle R^{M}(\omega^{\sharp}, N)\omega^{\sharp}, N\rangle\right]\nonumber\\
&-\mathrm{Hess}^Mf(N, N)|\omega|^{2}-\mathrm{Hess}^{M}f(\omega^{\sharp}, \omega^{\sharp}).\nonumber
\end{align}
Let $\lambda$ be the Einstein constant of $\left(P, g_{P}\right)$. Since
\begin{align*}
\mathrm{Ric}^{M}(N, N)=&\lambda|N^{P}|^2,\\
\mathrm{Ric}^{M}(\omega^{\sharp}, \omega^{\sharp})=&\lambda|\omega^{P}|^2,\\
\langle R^{M}(\omega^{\sharp}, N)\omega^{\sharp}, N\rangle=&\langle R^{P}(\omega^{P}, N^{P})\omega^{P}, N^{P}\rangle\\
=&K^{P}(\omega\wedge N^{P})|\omega^{P}\wedge N^{P}|^2\\
\leq&4\left(|\omega^{P}||N^{P}|^2-\langle\omega^{P}, N^{P}\rangle^2\right),
\end{align*}
we get that
\begin{align}\label{est1}
\mathcal{R}(\omega^{\sharp}, N)\leq& \frac{4}{3}\left[|\omega|^{2}|N^{P}|^2\left(k+2-\lambda-3|N^{P}|^2\right)+|\omega^{P}|^{2}\left(k+2-\lambda-|N^{P}|^2\right)\right.\\
&\left.\,\,\,\,\,-6\langle\omega^{P}, N^{P}\rangle^{2}+4|\omega^{P}|^{2}|N^{P}|^2\right]\nonumber\\
&-\lambda|N^{R}|^2|\omega|^{2}-\lambda|\omega^{R}|^{2}\nonumber\\
\leq&\frac{4}{3}\left[|\omega|^{2}|N^{P}|^{2}\left(k+4-\lambda-3|N^{P}|^{2}\right)+|\omega^{P}|^{2}\left(k+4-\lambda-|N^{P}|^{2}\right)\right]-\eta,\nonumber
\end{align}
where 
\[
\ \eta=8\langle\omega^{P}, N^{P}\rangle^{2}+\lambda|N^{R}|^{2}|\omega|^{2}+\lambda|\omega^{R}|^{2}\geq 0,
\]
and in the last step we have used that
\[
\ 4|\omega^{P}|^{2}|N^{P}|^{2}\leq2|\omega^{P}|^{2}+2|\omega|^2|N^{P}|^2.
\]
Recalling now that the Einstein constants of $\mathbb{HP}^{p}$, $4p=k$, and $C\mathrm{a}\mathbb{P}^{2}$ are respectively $\lambda=k+8$ and $\lambda=36$ we get that in these cases $\mathcal{R}(\omega^{\sharp}, N)<0$. From this fact and Corollary \ref{maincorfminimal} the index estimate given in Corollary \ref{coroindCROSSxR} immediately follows.

Let us now restrict our attention to te case  $P^{k}=\mathbb{CP}^n$, $2n=k\geq 4$. In this case the Einstein constant is $\lambda=k+2$, hence we cannot determine the sign of the RHS in \eqref{est1}. However, by \eqref{est1},
\begin{align*}
\mathcal{R}(\omega^{\sharp}, N)\leq& \frac{4}{3}\left[-3|\omega|^{2}|N^{P}|^4-|\omega^{P}|^{2}|N^{P}|^2-6\langle\omega^{P}, N^{P}\rangle^{2}+4|\omega^{P}|^{2}|N^{P}|^2\right]\nonumber\\
&-(k+2)|N^{R}|^2|\omega|^{2}-(k+2)|\omega^{R}|^{2}\\
\leq& 4|\omega^{P}|^{2}|N^{P}|^2-4|\omega|^{2}|N^{P}|^4-8\langle\omega^{P}, N^{P}\rangle^{2}-(k+2)|N^{R}|^2|\omega|^{2}-(k+2)|\omega^{R}|^{2}\\
\leq& 4|\omega|^{2}|N^{P}|^{2}\left(1-|N^{P}|^{2}\right)-8\langle\omega^{P},N^{P}\rangle^{2}-(k+2)|N^{R}|^{2}|\omega|^{2}-(k+2)|\omega^{R}|^{2}\\
=&|\omega|^{2}|N^{R}|^{2}\left(4|N^{P}|^{2}-(k+2)\right)-8\langle\omega^{P}, N^{P}\rangle^{2}-(k+2)|\omega^{R}|^{2}\\
\leq& 0.
\end{align*}
We now analyze the equality case. We begin by observing that $\mathcal{R}(\omega^{\sharp}, N)=0$ if and only if
\begin{itemize}
\item[(i)] $\omega^R=N^R=0$;
\item[(ii)] $4|\omega|^2=4|\omega^P|^2=K^{P}(\omega^P\wedge N^{P})=K^{P}(\omega^\sharp\wedge N)=|\omega|^2+3g_P(\omega,JN)^2$, where $J$ denotes the complex structure in $P=\mathbb{CP}^n$.
\end{itemize}
In particular, note that (ii) is satisfied if and only if $\omega^\sharp =h JN$, for some $h\in C^\infty(\Sigma)$.\\

Finally, note that if we are able to show that the previous conditions can be satisfied if and only if $\omega$ vanishes identically on $\Sigma$, then the conclusion of the theorem follows at once. This  is the content of the lemma below.  
\end{proof}

\begin{lemma}
Let $P^k=\mathbb{C}\mathbb{P}^n$ with $k=2n\geq 4$ and let $\Sigma$ be a closed isometrically immersed $f$-minimal hypersurface in the weighted manifold $M_{f}^{m+1}=\left(P^k\times\mathbb{R}^{m-k+1},g_M, e^{-f}d\mathrm{vol}_M\right)$, $g_M=g_{P}+dt_1^2+\dots+dt_{m-k+1}^2$, with $g_{P}$ denoting the canonical metric on $P$, $d\mathrm{vol}_M$ the Riemannian volume measure, and $f(x,t)=\frac{k+2}{2}|t|^2$. Let $\omega$ be any $f$-harmonic 1-form on $\Sigma$ and suppose that
\begin{itemize}
\item[(i)] $\omega^R=N^R=0$;
\item[(ii)] $\omega^\sharp =h JN$, for some $h\in C^\infty(\Sigma)$, where $J$ denotes the complex structure in $P=\mathbb{CP}^n$.
\end{itemize}
Then $\omega\equiv 0$.
\end{lemma}
\begin{proof}
\textbf{Step 1:} $h \nabla^M _{JN} N=\nabla^P_{JN} N=-J\nabla h$.\\
Since $\omega=\omega^P$ is a $f$-harmonic 1-form and $\nabla^M f=(k+2)T$, with $$T=t_1\frac{\partial}{\partial t_1}+\ldots+t_{m+k-1}\frac{\partial}{\partial t_{m+k-1}},$$ we have that 
\[
0=\delta_f \omega=\delta\omega +\langle \omega^\sharp,\nabla f\rangle=\delta \omega.
\]
In particular, $\omega$ is a harmonic 1-form and the conclusion of Step 1 follows reasoning, with minor modifications, as in \cite[Lemma A.1]{ACS}.\\
\medskip

\textbf{Step 2:} $\mathrm{Ric}_f(\omega^\sharp,\omega^\sharp)=(k-2)|\omega|^2-|\nabla h|^2$.\\
Note that, since $\omega^R=N^R=0$, we have
\begin{align*}
\mathrm{Hess} f (\omega^\sharp,\omega^\sharp)&=\mathrm{Hess}^M f(\omega^\sharp,\omega^\sharp)+\langle \nabla^M f,N\rangle\langle A\omega^\sharp,\omega^\sharp\rangle\\
&=(k+2)|\omega^R|^2+(k+2)\langle T,N^R\rangle\langle A\omega^\sharp,\omega^\sharp\rangle\\
&=0.
\end{align*}
Hence
\begin{align*}
\mathrm{Ric}_f(\omega^\sharp,\omega^\sharp)&=\mathrm{Ric}(\omega^\sharp,\omega^\sharp)\\
&=\mathrm{Ric}^P(\omega^\sharp,\omega^\sharp)-\langle R^P(\omega^\sharp,N)\omega^\sharp,N\rangle-|A\omega^\sharp|^2\\
&=(k+2)|\omega|^2-4|\omega|^2-|h \nabla^M _{JN} N|^2
\end{align*}
and Step 2 follows immediately from Step 1.\\
\medskip

\textbf{Step 3:} $|\nabla h|^2\leq |\nabla \omega|^2$.\\
Differentiating $\omega^\sharp=h JN$ we obtain, for any $X\in T\Sigma$,
\[
\nabla_X\omega^\sharp=h\nabla_X JN+\langle \nabla h,X\rangle JN.
\]
Since $\nabla_X JN$ and $JN$ are orthogonal, we immediately get
\[
|\nabla_X\omega^\sharp|^2=h^2|\nabla_X JN|^2+\langle \nabla h,X\rangle^2\geq \langle \nabla h,X\rangle^2.
\]
Step 3 follows at once.\\
\medskip 

\textbf{Step 4}: Conclusion.\\ 
Since $\omega$ is $f$-harmonic, by the $f$-Bochner-Weitzenb\"{o}ck formula and steps 2 and 3, we have that
\begin{align*}
-\frac{\Delta_f|\omega|^2}{2}&=|\nabla \omega|^2+\mathrm{Ric}_f(\omega^\sharp,\omega^\sharp)\\
&=|\nabla \omega|^2+(k-2)|\omega|^2-|\nabla h|^2\\
&\geq (k-2)|\omega|^2.
\end{align*} 
Integrating over $\Sigma$ we conclude that
\[
(k-2)\int_\Sigma |\omega|^2 e^{-f}d\mathrm{vol}_\Sigma\leq 0.
\]
Since $k>2$, we conclude that $\omega$ must vanish identically on $\Sigma$.
\end{proof}

\begin{acknowledgement*}
The first author is partially supported by INdAM-GNSAGA. The second author is partially supported by INdAM-GNAMPA.
\end{acknowledgement*}

\end{document}